\documentclass[a4paper,reqno]{amsart}
      \usepackage{etex}
      \usepackage[OT2, T1]{fontenc}
      \usepackage{url}
      \usepackage{amsmath}
      \usepackage{lmodern}
      \usepackage{array}
      \usepackage{graphicx}
      \usepackage{amsfonts}
      \usepackage{amssymb}
      \usepackage{extpfeil}
\usepackage{soul}
\usepackage{xargs}
      \usepackage{xcolor}
      \definecolor{royalblue}{RGB}{64, 106, 212}
      \definecolor{link}{RGB}{11,0,128}
      \definecolor{olivegreen}{RGB}{128, 128, 0}
      \usepackage{aliascnt}
      \usepackage[colorlinks=true, citecolor=royalblue,linkcolor=olivegreen,urlcolor=link]{hyperref}
      \usepackage{amsthm}
      \usepackage{esint}
      \usepackage{epic}
      \usepackage{rotating}

      \usepackage{calligra}
      \usepackage{xkeyval}

      \usepackage{exscale, relsize}
      \usepackage{stackengine}
      \usepackage{scalerel}
      \usepackage{setspace}
      \usepackage{paralist}
      \usepackage{colonequals}		
      \usepackage{mathabx}		
      \usepackage[left]{lineno}
      \usepackage{amscd}
      \usepackage[all,cmtip]{xy}	
      \usepackage{sseq}			
      \usepackage{verbatim}
      \usepackage{parskip}
      \usepackage{microtype}
      \usepackage{ragged2e}
            \usepackage[in]{fullpage}
      \usepackage{graphicx}
      \usepackage{mathrsfs} 			
      \usepackage{bm} 				
      \usepackage{cleveref}
      \usepackage{enumitem}
      \usepackage{moreenum}			
      \usepackage{blindtext}
      \usepackage{tikz-cd}
      \usepackage{tikz}
      \usepackage{subfiles}
      \usetikzlibrary{shapes.geometric}
      \usetikzlibrary{automata, bending,matrix,arrows,decorations.pathmorphing,backgrounds,positioning,fit,petri,arrows.meta}
      \usepackage{textcomp}
      \usepackage{indentfirst}
      \usetikzlibrary{arrows}
      \usepackage{enumitem}
      \tikzset{commutative diagrams/.cd,arrow style=tikz,diagrams={>=latex'}}
      \usepackage{filecontents}
      \usepackage[alphabetic,lite]{amsrefs} 	
      \usepackage{stmaryrd}	
      \usepackage{subfiles}
      \usepackage{ragged2e}
      \usetikzlibrary{arrows}
      \usepackage{stackengine}

\usepackage[colorinlistoftodos,prependcaption,textsize=tiny]{todonotes}
\newcommandx{\unsure}[2][1=]{\todo[linecolor=red,backgroundcolor=red!25,bordercolor=red,#1]{#2}}
\newcommandx{\change}[2][1=]{\todo[linecolor=blue,backgroundcolor=blue!25,bordercolor=blue,#1]{#2}}
\newcommandx{\info}[2][1=]{\todo[linecolor=OliveGreen,backgroundcolor=OliveGreen!25,bordercolor=OliveGreen,#1]{#2}}
\newcommandx{\improvement}[2][1=]{\todo[linecolor=Plum,backgroundcolor=Plum!25,bordercolor=Plum,#1]{#2}}
\newcommandx{\thiswillnotshow}[2][1=]{\todo[disable,#1]{#2}}

\usepackage{marginnote}
\newcommand{\mytodo}[2][]{{%
 \let\marginpar\marginnote
 \reversemarginpar
 \renewcommand{\baselinestretch}{0.8}%
 \todo[#1]{#2}}}

         \newcommand{\bA}{\mathbb{A}}

         \newcommand{\bG}{\mathbb{G}}

         \newcommand{\bP}{\mathbb{P}}

         \newcommand{\cE}{\mathcal{E}}

         \newcommand{\cO}{\mathcal{O}}

         \newcommand{\fm}{\mathfrak{m}}

         \newcommand{\sG}{\mathscr{G}}

         \newcommand{\ra}{\rightarrow}

         \newcommand{\pr}{^{\prime}}

         \renewcommand{\b}{\textbf}

         \providecommand{\fps}[1]{[\![#1]\!]}
         \providecommand{\lps}[1]{(\!(#1)\!)}
         \providecommand{\SP}[1]{\cite[\href{https://stacks.math.columbia.edu/tag/#1}{#1}]{SP}}

         \providecommand{\fps}[1]{\llbracket#1\rrbracket}
         \providecommand{\lps}[1]{(\!(#1)\!)}

\newextarrow{\xbigtoto}{{15}{15}{15}{12}}
   {\bigRelbar\bigRelbar{\bigtwoarrowsleft\rightarrow\rightarrow}}
         
         \DeclareMathOperator{\Spec}{Spec}		                       
         \DeclareMathOperator{\rad}{rad}			                       
         \DeclareMathOperator{\Frac}{Frac}		                       
         \DeclareMathOperator{\GL}{GL}		                                                  
         \DeclareMathOperator{\rk}{rk}		                                                  
         \DeclareMathOperator{\codim}{codim}		                                                  

         \newcommand{\ba}{\begin{aligned}}
         \newcommand{\ea}{\end{aligned}}
         \newcommand{\be}{\begin{equation}}
         \newcommand{\ee}{\end{equation}}
         \newcommand{\pf}{\begin{proof}}
         \newcommand{\bpf}{\begin{proof}}
         \newcommand{\epf}{\end{proof}}
         \newcommand{\bsol}{\begin{solution}}
         \newcommand{\esol}{\end{solution}}
         \newcommand{\bthm}{\begin{thm}}
         \newcommand{\ethm}{\end{thm}}

         \newcommand{\bprop}{\begin{prop}}
         \newcommand{\eprop}{\end{prop}}
         \newcommand{\bcor}{\begin{cor}}
         \newcommand{\ecor}{\end{cor}}

         \newcommand{\brem}{\begin{rem}}
         \newcommand{\erem}{\end{rem}}

         \newcommand{\brems}{\begin{rems} \hfill \begin{enumerate}[label=\b{\thenumberingbase.},ref=\thenumberingbase]}
         
         \newcommand{\erems}{\end{enumerate} \end{rems}}
         \newcommand{\begs}{\begin{egs} \hfill \begin{enumerate}[label=\b{\thenumberingbase.},ref=\thenumberingbase]}
         
         \newcommand{\eegs}{\end{enumerate} \end{egs}}
         \newcommand{\eremstweak}{\end{enumerate} \end{rems-tweak}}
         \newcommand{\eremst}{\end{enumerate} \end{rems-tweak}}
         \newcommand{\blem}{\begin{lemma}}
         \newcommand{\elem}{\end{lemma}}

         \newcommand{\bconj}{\begin{conj}}
         \newcommand{\econj}{\end{conj}}
         \newcommand{\bprob}{\begin{Problem}}
         \newcommand{\eprob}{\end{Problem}}

         \newcommand{\bq}{\begin{Q}}
         \newcommand{\eq}{\end{Q}}
         \newcommand{\benum}{\begin{enumerate}[label={{\upshape(\alph*)}}]}
         \newcommand{\benuma}{\begin{enumerate}[label={{\upshape(\arabic*)}}]}
         \newcommand{\benumr}{\begin{enumerate}[label={{\upshape(\roman*)}}]}
         \newcommand{\eenum}{\end{enumerate}}
         \newcommand{\bc}{}
         \newcommand{\bd}{\begin{defn}}
         \newcommand{\ed}{\end{defn}}
         \newcommand{\bque}{\begin{que}}
         \newcommand{\eque}{\end{que}}
         \newcommand{\bfct}{\begin{fact}}
         \newcommand{\efct}{\end{fact}}

         \newcommand{\beg}{\begin{eg}}
         \newcommand{\eeg}{\end{eg}}

         \newcommand{\bcl}{\begin{claim}}
         \newcommand{\ecl}{\end{claim}}

\tikzset{
    labl/.style={anchor=south, rotate=90, inner sep=.5mm}
}

\newaliascnt{numberingbase}{subsection}
\numberwithin{equation}{numberingbase}

\newcommand{\statementskip}{9pt plus 2pt minus 1pt}
\newcommand{\subsectionbeforeskip}{9pt plus 2pt minus 1pt}
\newcommand{\subsectionafterskip}{4pt plus 1pt minus .5pt}
\makeatletter
\def\subsection{\@startsection{subsection}{2}%
  \z@{\subsectionbeforeskip}{\subsectionafterskip}%
  {\normalfont\bfseries}}
\def\subsubsection{\@startsection{subsubsection}{3}%
  \z@{7pt plus 1.5pt minus 1pt}{3pt plus 1pt minus .5pt}%
  {\normalfont\itshape}}
\makeatother

\newtheoremstyle{thms}{\statementskip}{\statementskip}{\itshape}{}{\bfseries}{.}{ }{}
\theoremstyle{thms}
\newaliascnt{conj}{numberingbase}
\newtheorem{conj}[conj]{Conjecture}
\aliascntresetthe{conj}
\newaliascnt{corollary}{numberingbase}
\newtheorem{corollary}[corollary]{Corollary}
\aliascntresetthe{corollary}
\newaliascnt{cor}{numberingbase}
\newtheorem{cor}[cor]{Corollary}
\aliascntresetthe{cor}
\newaliascnt{lemma}{numberingbase}
\newtheorem{lemma}[lemma]{Lemma}
\aliascntresetthe{lemma}
\newaliascnt{sublemma}{equation}

\aliascntresetthe{sublemma}
\newaliascnt{prop}{numberingbase}
\newtheorem{prop}[prop]{Proposition}
\aliascntresetthe{prop}
\newaliascnt{proposition}{numberingbase}
\newtheorem{proposition}[proposition]{Proposition}
\aliascntresetthe{proposition}
\newaliascnt{Q}{numberingbase}
\newtheorem{Q}[Q]{Question}
\aliascntresetthe{Q}
\newaliascnt{thm}{numberingbase}
\newtheorem{thm}[thm]{Theorem}
\aliascntresetthe{thm}
\newaliascnt{theorem}{numberingbase}
\newtheorem{theorem}[theorem]{Theorem}
\aliascntresetthe{theorem}
\newaliascnt{variant}{numberingbase}

\aliascntresetthe{variant}

\newtheoremstyle{claims}{\statementskip}{\statementskip}{}{}{\itshape}{.}{ }{}
\theoremstyle{claims}
\newaliascnt{claim}{equation}
\newtheorem{claim}[claim]{Claim}
\aliascntresetthe{claim}
\newaliascnt{cl-tweak}{subsubsection}
\newtheorem{cl-tweak}[cl-tweak]{Claim}
\aliascntresetthe{cl-tweak}
\Crefname{cl-tweak}{Claim}{Claims}

\newtheoremstyle{defs}{\statementskip}{\statementskip}{}{}{\bfseries}{.}{ }{}
\theoremstyle{defs}
\newaliascnt{defn}{numberingbase}
\newtheorem{defn}[defn]{Definition}
\aliascntresetthe{defn}
\newaliascnt{definition}{numberingbase}

\aliascntresetthe{definition}
\newaliascnt{eg}{numberingbase}
\newtheorem{eg}[eg]{Example}
\aliascntresetthe{eg}
\newaliascnt{example}{numberingbase}

\aliascntresetthe{example}
\newtheorem*{egs}{Examples}
\newaliascnt{rem}{numberingbase}
\newtheorem{rem}[rem]{Remark}
\aliascntresetthe{rem}
\newaliascnt{remark}{numberingbase}
\newtheorem{remark}[remark]{Remark}
\aliascntresetthe{remark}
\newtheorem*{rems}{Remarks}

\Crefname{claim}{Claim}{Claims}
\Crefname{sublemma}{Lemma}{Lemmas}
\Crefname{conj}{Conjecture}{Conjectures}
\Crefname{cor}{Corollary}{Corollaries}
\Crefname{defn}{Definition}{Definitions}
\Crefname{eg}{Example}{Examples}
\Crefname{prop}{Proposition}{Propositions}
\Crefname{Q}{Question}{Questions}
\Crefname{rem}{Remark}{Remarks}
\Crefname{thm}{Theorem}{Theorems}
\Crefname{variant}{Variant}{Variants}
\crefname{claim}{claim}{claims}
\crefname{sublemma}{lemma}{lemmas}
\crefname{conj}{conjecture}{conjectures}
\crefname{cor}{corollary}{corollaries}
\crefname{corollary}{corollary}{corollaries}
\crefname{defn}{definition}{definitions}
\crefname{definition}{definition}{definitions}
\crefname{eg}{example}{examples}
\crefname{example}{example}{examples}
\crefname{lemma}{lemma}{lemmas}
\crefname{prop}{proposition}{propositions}
\crefname{proposition}{proposition}{propositions}
\crefname{Q}{question}{questions}
\crefname{rem}{remark}{remarks}
\crefname{remark}{remark}{remarks}
\crefname{thm}{theorem}{theorems}
\crefname{theorem}{theorem}{theorems}
\crefname{variant}{variant}{variants}
\Crefname{corollary}{Corollary}{Corollaries}
\Crefname{definition}{Definition}{Definitions}
\Crefname{example}{Example}{Examples}
\Crefname{lemma}{Lemma}{Lemmas}
\Crefname{proposition}{Proposition}{Propositions}
\Crefname{remark}{Remark}{Remarks}
\Crefname{theorem}{Theorem}{Theorems}

\theoremstyle{thms}
\newaliascnt{thm-tweak}{subsection}
\newtheorem{thm-tweak}[thm-tweak]{Theorem}
\aliascntresetthe{thm-tweak}
\crefname{thm-tweak}{theorem}{theorems}
\Crefname{thm-tweak}{Theorem}{Theorems}
\newaliascnt{lemma-tweak}{subsection}
\newtheorem{lemma-tweak}[lemma-tweak]{Lemma}
\aliascntresetthe{lemma-tweak}
\crefname{lemma-tweak}{lemma}{lemmas}
\Crefname{lemma-tweak}{Lemma}{Lemmas}
\newaliascnt{cor-tweak}{subsection}
\newtheorem{cor-tweak}[cor-tweak]{Corollary}
\aliascntresetthe{cor-tweak}
\crefname{cor-tweak}{corollary}{corollaries}
\Crefname{cor-tweak}{Corollary}{Corollaries}
\newaliascnt{prop-tweak}{subsection}
\newtheorem{prop-tweak}[prop-tweak]{Proposition}
\aliascntresetthe{prop-tweak}
\crefname{prop-tweak}{proposition}{propositions}
\Crefname{prop-tweak}{Proposition}{Propositions}
\newaliascnt{conj-tweak}{subsection}
\newtheorem{conj-tweak}[conj-tweak]{Conjecture}
\aliascntresetthe{conj-tweak}
\crefname{conj-tweak}{conjecture}{conjectures}
\Crefname{conj-tweak}{Conjecture}{Conjectures}

\theoremstyle{defs}
\newaliascnt{defn-tweak}{subsection}
\newtheorem{defn-tweak}[defn-tweak]{Definition}
\aliascntresetthe{defn-tweak}
\crefname{defn-tweak}{definition}{definitions}
\Crefname{defn-tweak}{Definition}{Definitions}
\newaliascnt{eg-tweak}{subsection}
\newtheorem{eg-tweak}[eg-tweak]{Example}
\aliascntresetthe{eg-tweak}
\crefname{eg-tweak}{example}{examples}
\Crefname{eg-tweak}{Example}{Examples}
\newtheorem*{rems-tweak}{Remarks}
\newaliascnt{rem-tweak}{subsection}
\newtheorem{rem-tweak}[rem-tweak]{Remark}
\aliascntresetthe{rem-tweak}
\crefname{rem-tweak}{remark}{remarks}
\Crefname{rem-tweak}{Remark}{Remarks}

\newtheoremstyle{subsection-tweak}
   {\subsectionbeforeskip}
   {\subsectionafterskip}%
   {}
   {}%
   {\bfseries}
   {}%
   {\newline}
   {\thmnumber{\@{#1}{}\@{#2}.}%
    \thmnote{~{\bfseries#3.}}}

\makeatletter
\newcommand{\ppafterheading}{%
  \@ifnextchar\label{\ppafterheadinglabel}{\ignorespaces}%
}
\def\ppafterheadinglabel\label#1{\label{#1}\ignorespaces}
\makeatother

\newenvironment{pp}[1][]{%
  \par\addvspace{\subsectionbeforeskip}%
  \refstepcounter{subsection}%
  \noindent\textbf{\thesubsection.\if\relax\detokenize{#1}\relax\else\ #1.\fi}%
  \par\nobreak\vspace{\subsectionafterskip}\noindent\ppafterheading
}{%
  \par\addvspace{\subsectionafterskip}%
}
\newcommand{\bpp}{\begin{pp}}
\newcommand{\epp}{\end{pp}}
\newaliascnt{pp-t}{subsubsection}
\newtheorem{pp-t}[pp-t]{}
\aliascntresetthe{pp-t}
\crefname{pp-t}{section}{sections}
\Crefname{pp-t}{Section}{Sections}

\theoremstyle{subsection-tweak}
\newaliascnt{conventions}{subsection}

\aliascntresetthe{conventions}
\crefname{conventions}{section}{sections}
\Crefname{conventions}{Section}{Sections}

\theoremstyle{subsection-tweak}
\newtheorem{pp-tweak}{}
\crefname{pp-tweak}{section}{sections}
\Crefname{pp-tweak}{Section}{Sections}

      \makeatletter
      \def\@tocline#1#2#3#4#5#6#7{
          \begingroup
          \@ifempty{#4}{}{}

          \parindent\z@ \leftskip#3\relax \advance\leftskip\@tempdima\relax
          #5\hskip-\@tempdima
            \ifcase #1
             \or\or \hskip 2em \or \hskip 1em \else \hskip 3em \fi%
            #6\nobreak\relax
          \dotfill\hbox to\@pnumwidth{\@tocpagenum{#7}}\par
          \nobreak
          \endgroup
        }
       \def\l@section{\@tocline{1}{0pt}{1pc}{}{}}

      \renewcommand{\tocsection}[3]{%
        \indentlabel{\@ifnotempty{#2}{\makebox[1.3em][l]{%
          \ignorespaces#1 \bfseries{#2}.\hfill}}}\bfseries{#3}
          \vspace{-3.5pt}}

      \renewcommand{\tocsubsection}[3]{%
        \indentlabel{\@ifnotempty{#2}{\hspace*{-0.5em}\makebox[2.1em][l]{%
          \ignorespaces#1#2.\hfill}}}#3
          \vspace{-4.5pt}}

\makeatother

\makeatletter 
\newcommand\appendix@section[1]{%
  \refstepcounter{section}%
  \orig@section*{Appendix \@Alph\c@section. #1}%
}
\let\orig@section\section
\g@addto@macro\appendix{\let\section\appendix@section}
\makeatother

      \DeclareMathVersion{normal2}

\def\UTFviii@defined#1{%
  \ifx#1\relax
      \PackageError{inputenc}{Unicode\space char\space\expandafter
                              \UTFviii@splitcsname\string#1\relax
                              \MessageBreak
                              not\space set\space up\space
                              for\space use\space with\space LaTeX}\@eha
  \else\expandafter
    #1%
  \fi
}

\makeatletter
\def\UTFviii@defined#1{%
  \ifx#1\relax
      ?%
  \else\expandafter
    #1%
  \fi
}

\makeatother

\providecommand{\A}{\mathbb A}
\providecommand{\Pj}{\mathbb P}
\providecommand{\Gm}{{\mathbb G_m}}
\providecommand{\OO}{\mathcal O}
\providecommand{\Max}{\operatorname{Max}}
\providecommand{\Hone}{H^1}
\DeclareMathDelimiter{\llbracket}{\mathopen}{stmry}{"4A}{stmry}{"71}
\DeclareMathDelimiter{\rrbracket}{\mathclose}{stmry}{"4B}{stmry}{"79}

\author{Fei Liu}
\address{Department of Mathematics, Run Run Shaw Building, The University of Hong Kong, Hong Kong}
\email{liufei54@pku.edu.cn}
\date{September 10, 2026}

\title{The unramified Grothendieck--Serre conjecture}

\begin{document}
\begin{abstract}
We prove that a generically trivial torsor under a reductive group
scheme over an unramified regular local ring is trivial. 
\end{abstract}
\maketitle
\vspace{-25pt}
\setcounter{tocdepth}{1}
\begingroup
\hypersetup{linkcolor=black}
\tableofcontents
\endgroup
\medskip

\section{Introduction}\label{sec:introduction}

This article concerns the following conjecture of Grothendieck and
Serre \cite{Ser58}*{page~31, remarque},
\cite{Gro58}*{pages~26--27, remarques~3},
\cite{Gro68}*{remarques~1.11~a)} about the triviality of torsors.

\begin{conj}[Grothendieck--Serre]\label{conj:gs}
For a reductive group scheme $G$ over a regular local ring
$R$, no nontrivial $G$-torsor over $R$ trivializes over
its fraction field, that is,
\[
 \ker\bigl(\Hone(R,G)\ra \Hone(\Frac R,G)\bigr)=\{*\}.
\]
\end{conj}

The equicharacteristic case was proved by Fedorov--Panin \cite{FP15}, while the mixed-characteristic case remains widely open, except for several unramified cases, namely, those of constant groups \cite{GL23} and totally isotropic groups \cite{CF23}.
In this article, we settle \Cref{conj:gs} for all unramified regular local rings.

\begin{theorem}\label{thm:main}
For a Noetherian semilocal ring $R$ that is flat and geometrically
regular\footnote{Equivalently, by Popescu's theorem \cite{SP}*{Theorem 07GC}, $R$ is ind-smooth over $\cO$.} over a Dedekind ring $\OO$, and a reductive $R$-group
scheme $G$, every generically trivial $G$-torsor is trivial; equivalently,
\begin{equation}
 \Hone(R,G)\hookrightarrow \Hone(\Frac R,G) \tag{$\bigstar_{\!G}$}.
\end{equation}
\end{theorem}
The equivalence follows from the twisting bijections in nonabelian cohomology: the injectivity of ($\bigstar_{\!G}$) amounts to the assertion that the kernel of ($\bigstar_{\!G'}$) is trivial for all strongly inner forms $G'$ of $G$; see  \cite{Gir71}*{Chapitre III, Proposition 2.6.1 (i)}\textcolor{red}{.}
\begin{corollary}\label{cor:local}
Let $G$ be a reductive group scheme over an unramified regular local ring $R$. Then every
generically trivial $G$-torsor over $R$ is  trivial.
\end{corollary}

We recall that a regular local ring ($R,\fm$) is \emph{unramified} if either $R$ contains a field or 
 $p\notin\mathfrak m^2$ with $p:= \mathrm{char}(R/\fm)$. In the former case $R$ is flat and geometrically regular over its prime subfield, while in the latter case $R$ is flat and geometrically regular over
$\mathbb Z_{(p)}$, so the corollary follows from \Cref{thm:main}.

The Grothendieck--Serre conjecture has numerous concrete consequences, a few of which we now illustrate.
The first application concerns the uniqueness of reductive models over regular semilocal rings.

\begin{theorem}\label{red-models}
Let $\cO$ be a Dedekind ring, and let $R$ be a flat, geometrically regular, Noetherian semilocal $\cO$-algebra.
Let $G, G'$ be reductive $R$-groups. Then $G \simeq G'$ iff $G_{\Frac{R}} \simeq G_{\Frac{R}}'$. 
\end{theorem}
In particular, a reductive $R$-group $G$ is split iff $G_{\Frac R}$ is split. This case was already known \cite{Ces22a}*{Theorem 1.2} and in fact follows from the split case of the Grothendieck-Serre conjecture.

As a second application, we have the following product formula, which was obtained in \cite{Ces22a} for $G$ quasi-split.

\begin{theorem}\label{prod-formula}
 For a Noetherian semilocal ring $R$ that is flat and geometrically
regular over a Dedekind ring $\OO$, an $r\in \fm_R \backslash \{0\}$, and the $r$-adic completion $\widehat{R}$,
\[
G(\widehat{R}[\tfrac{1}{r}])=G(\widehat{R}) G(R[\tfrac{1}{r}]) \quad \text{ for every reductive $R$-group $G$}.
\]
\end{theorem}
In fact, by Beauville-Laszlo-type formal gluing (see, e.g., \cite{BC22}*{Lemma 2.2.11 (b)} or \cite{Fed16}*{Proposition~4.4}), the double coset space
$ G(\widehat{R}) \backslash G(\widehat{R}[\tfrac{1}{r}])/G(R[\tfrac{1}{r}])$ is naturally in bijection with the isomorphism classes of $G$-torsors over $R$ that trivialize over both $R[\tfrac{1}{r}]$ and $\widehat{R}$, which are all trivial by \Cref{thm:main}.

We note that the case $\dim R\le 1$ of \Cref{conj:gs} is proved by establishing the corresponding product formula using deep group-theoretic results (notably, Bruhat-Tits theory), whereas the product formula for general $R$ seems to resist any direct attack.

The next result concerns a conjecture attributed to Colliot-Th\'el\`ene \cite{CT79} and Panin \cite{Pan09}: a reductive group scheme over a regular local ring should have a parabolic subgroup of a given type whenever its generic fibre does. Further results for quadratic forms include \cite{PP10}, \cite{PP15}, \cite{Scu18}. For Borel subgroups, this was reduced to the Grothendieck--Serre conjecture \cite{Ces22a}*{Theorem~9.5}.

\begin{theorem}
Let $\cO$ be a Dedekind ring, and let $R$ be a flat, geometrically regular, Noetherian semilocal $\cO$-algebra.
Let $G$ be a reductive $R$-group scheme. Then $G$ is quasi-split iff $G_{\Frac{R}}$ is quasi-split.
\end{theorem}

\bpp[History]
\label{history}
Since it was proposed by Serre and Grothendieck, \Cref{conj:gs} has been studied extensively, and the following cases have been established.
\benumr
\item The case when $G$ is a torus was proved by Colliot-Th\'el\`ene and Sansuc in \cite{CTS87}.
\item The case when $\dim R\le 1$ was established in \cite{Guo22}, building on Nisnevich's work \cite{Nis82} and \cite{Nis84}, with valuation ring variants treated in \cite{Guo24} and \cite{GL23}*{Appendix~A}\textcolor{red}{.}
Several subcases were proved in \cite{Har67}, \cite{BB70}, \cite{BrT3}, \cite{PS16}, \cite{BVG14}, \cite{BFF17}, \cite{BFFH20}. This
one-dimensional case implies the case when $R$ is Henselian\textcolor{red}{;} see \cite{BB70} and \cite{CTS79}*{assertion~6.6.1}.

\item The equicharacteristic case, that is, when $R$ contains a field $k$, was settled by Fedorov--Panin \cite{FP15} and Panin \cite{Pan20}, with simplifications in \cite{Fed22a}. 
Prior to these works, significant subcases were proven in \cite{Oja80}, \cite{CTO92}, \cite{Rag94}, \cite{PS97}, \cite{Zai00}, \cite{OP01}, \cite{OPZ04}, \cite{Pan05}, \cite{Zai05}, \cite{Che10}, \cite{PSV15}, \cite{Pan20b}. Panin \cite{Pan22a} also proved variants beyond connected reductive groups.
Recently, Bouthier--\v{C}esnavi\v{c}ius--Scavia \cite{BCS25} proved the variant of \Cref{conj:gs} for all smooth affine $k$-groups $G$ and essentially smooth semilocal $k$-algebras $R$. 
\item In mixed characteristic, \v{C}esnavi\v{c}ius \cite{Ces22a}*{Theorem~9.1} settled the quasi-split unramified case.
 Fedorov \cite{Fed21} previously proved the split case under additional assumptions.
 Later, Guo and the author \cite{GL23} proved the case when $G$ descends to a reductive $\cO$-group, allowing, more generally, $\cO$ to be a semilocal Pr\"ufer ring (see also \cite{GL24}).
\v{C}esnavi\v{c}ius--Fedorov \cite{CF23} (after Fedorov's result \cite{FedMRL}) proved the unramified case when $G^{\mathrm{ad}}$ is totally isotropic.

\item Several sporadic cases where either $R$ or $G$ is special were established in  \cite{Gro68}*{Remarque~1.11~a)}, \cite{Oja82}, \cite{Nis89}, \cite{Pan22b}, \cite{Fir23}, \cite{BFFP22}, \cite{Pan19b}.
\eenum 
\epp

\begin{pp}[Outline of the proof]
The geometric approach to the Grothendieck-Serre conjecture was pioneered in \cite{CTO92}, carried through in \cite{FP15} in the equicharacteristic case, and developed much further in \cite{Fed21}, \cite{Ces22}, \cite{GL23}, \cite{CF23} in mixed characteristic.  Let $E$ be a generically trivial $G$-torsor over $R$ that we wish to trivialize; by Popescu's theorem \cite{SP}*{Theorem 07GC}, we may assume that $R$ is a semilocal ring of an $\cO$-smooth affine scheme. The core idea of this approach is to develop Gabber--Quillen-type presentation lemmas to study the geometry of $\Spec R$ and then use patching arguments to eventually find a $G$-torsor $\cE$ over $\bP_R^1$  such that $\cE|_{0_R}\simeq E$ and $\cE|_{\infty_R}\simeq 1$. 
We can then apply known techniques to prove that $E$ is trivial (for instance, by applying the sectionwise constancy result). 
 
 However, in contrast to the equicharacteristic situation, one immediately encounters a serious difficulty in mixed characteristic: the loss of one-dimensional flexibility in geometric arguments.  
 For instance, to apply the geometric approach above we have to build
a closed $Z\subset \Spec R$ of codimension $\ge 2$ away from which our generically trivial torsor $E$ is ``simpler'', while in equicharacteristic (namely when $\cO$ is a field), codimension $\ge 1$ suffices.
 
 The main new idea of this paper is to add one more dimension to restore the flexibility to ``move''; that is, we will work with smooth surfaces over $R$ and torsors over them. The main stages of the proof are as follows:
 
 \begin{itemize}
 \item [(1)]
 We construct an affine open $C\subset\A^1_R$
 , a section
$s\in C(R)$, an $R$-finite closed subscheme $Z\subset C$, and a $G$-torsor $Q$ that is defined over an open of
$\Pj^1_C$ containing
\[
 \Pj^1_{C\setminus Z},\qquad 0_C,\qquad\infty_C,
\]
with $Q|_{0_C}\simeq E_C$ and $s^*E_C\simeq E$, together with a
trivialization at the infinity section. The two factors in
$\Pj^1_C=C\times_R\Pj^1_R$ have different roles: $C$
is the curve on which we want to construct torsors, while the second factor
parametrizes the given family.
This construction occupies Sections 4--6 and starts with (a weak version of) the partial family supplied by \v{C}esnavi\v{c}ius--Fedorov, as
recalled in Proposition~\ref{prop:partial-family}.

 \item [(2)] Given the family in (1), we can pull it back along the graph of any $C$-morphism $C\to \bP_R^1$ whose image avoids the missing locus of the family. Concretely, we choose disjoint finite relative divisors $H,D\subset\Pj^1_R$
   such that $H\subset C$ contains $Z\cup s$, while $D$ contains the boundary of $C$, including the infinity section. We can then consider the pullback along the graph of the morphism \[
 j=[F_H^{d}:F_D^{h}]\colon\Pj^1_R\to\Pj^1_R,
\] 
where $ F_H\in H^0(\bP_R^1,\cO(h))$ and $F_D\in H^0(\bP_R^1,\cO(d))$ define $H$ and $D$, respectively.  Note that $j$ sends $H$ to zero and $D$ to infinity. Its graph over
$C_0=\Pj^1_R\setminus D$ lies in the domain of $Q$:
away from $Z$ the whole parameter line is available, and over
$Z$ the graph lies on its zero section. The graph pullback
therefore gives a torsor over $C_0$ whose restriction at
$s$ is $E$; a priori, the pullback is not defined on the entire projective line $\bP_R^1$.
 
 Now the key observation is that the graph pullback to $C$ indeed extends to a torsor over the entire $\bP_R^1$ (while keeping track of the relevant trivialization and $s$-pullback) if one replaces $j$ by its postcomposition with the power map 
 \[
 \varphi_N: \bP_R^1 \to \bP_R^1, \quad \varphi_N([x:y])=[x^N:y^N]  \quad \text{ for }  \quad N\gg 0;
 \]
  see Proposition~\ref{prop:graph}. This is made possible by the dilation result in \Cref{lem:dilation}. This is reminiscent of an idea of Quillen \cite{Qui76}: for any $f\in A$ and any ``nice'' $A$-group $G$, if a $G$-torsor over $A_f[u]$ is trivialized modulo $u$, then for $N\gg 0$ its base change along the dilation map $A_f[u] \to A_f[u], u\mapsto f^Nu$ should extend to a $G$-torsor over $A[u]$, which admits a trivialization modulo $u$ (in fact, we need a trivialization modulo $uf$ for our purpose) that agrees with the prescribed trivialization over $A_f$.

 These constructions give a $G$-torsor $\mathcal F$ over $\Pj^1_R$ with
$
 s^*\mathcal F\simeq E, \infty_R^*\mathcal F\simeq1.
$
We can now conclude that $E\simeq 1$ by the sectionwise constancy result (see \Cref{thm:sectionwise-constancy}).
 \end{itemize}

\end{pp}

\begin{pp}[Notation and conventions]
Throughout this paper, we work with commutative rings with units. For a
point $s$ of a scheme (resp., for a prime ideal $\mathfrak{p}$ of a ring), we let $k_s$ (resp., $k_{\mathfrak{p}}$) denote its residue
field. For a ring $A$, we let $\Frac A$ denote its total ring of fractions.
For a morphism of schemes $S\pr\to S$, we denote the base change functor from $S$ to $S\pr$ by $(-)_{S\pr}$. If $S\pr=\Spec R\pr$ is affine, we also write $(-)_{R\pr }$ for $(-)_{S\pr}$.

All torsors are right torsors, and $1$ denotes the trivial
torsor. We write $\mathbf{B}G$ for the classifying stack of $G$.
We use fppf cohomology throughout; for smooth group schemes,
such as reductive group schemes, it agrees with \'etale cohomology.
Trivializations used in gluing are chosen as part
of the data. 
 A semisimple group scheme is \emph{simple} if its geometric
fibres are nontrivial almost simple algebraic groups. 
\end{pp}

\subsection*{Acknowledgements}
I thank Professor Chun-Yin Hui for his support and constant encouragement.

\section{Preliminary results}\label{sec:projective-line-families}

The following weak form of \cite{CF23}*{Proposition~2.8} is our starting point: it supplies
a family of torsors over a ``large'' open of the projective line.

\begin{lemma}
\label{prop:partial-family}
Let $R$ be a Noetherian semilocal ring that is flat and
geometrically regular over a Dedekind ring, let $G$ be a
reductive $R$-group scheme, and let $E$ be a generically
trivial $G$-torsor. There are an open $V\subset\Pj^1_R$,
containing $0_R$, $\infty_R$, and every point of codimension
$\le 2$, and a $G$-torsor $Q$ over $V$ such that
\[
 Q|_{0_R}\simeq E,\qquad Q|_{\infty_R}\simeq1.
\]
\end{lemma}
Notice that the reduced complement $B=(\Pj^1_R\setminus V)_{\mathrm{red}}$ is $R$-finite because it
is proper over $R$ and contained in the affine open $\bA^1_R=\bP^1_R\backslash \{\infty_R\}$.

The following Horrocks-type statement is useful for recognizing
torsors that are pulled back from the base. It strengthens earlier versions in the literature, although these already suffice for our purposes; see \cite{PSV15}*{Proposition~9.6}, \cite{Tsy19}, \cite{Fed21}*{Proposition~2.2}, \cite{Ces22a}*{Lemma~8.3}.

\begin{lemma}[\cite{CF23}*{Proposition 3.1}]
\label{prop:closed-fibre-descent}
Let $G$ be a reductive group scheme over a semilocal ring $A$, and let $Q$ be a $G$-torsor over
$\Pj^1_A$. Then $Q$ descends to a $G$-torsor
over $A$ if and only if, for every maximal ideal
$\fm\subset A$, the restriction
$Q|_{\Pj^1_{k_{\fm}}}$ descends to a
$G$-torsor over $k_{\fm}$.
\end{lemma}

In concrete applications, one modifies a torsor over $\bP_A^1$ to make it trivial on all closed $A$-fibres;
then the restriction to the infinity section identifies the descended torsor.

The following sectionwise constancy result was established by Panin--Stavrova \cite{PS25}, and reproved and mildly generalized by Fedorov--\v{C}esnavi\v{c}ius.
\begin{theorem}[\cite{CF23}*{Theorem~3.5}]
\label{thm:sectionwise-constancy}
Let $A$ be a semilocal ring and let $G$ be a reductive
$A$-group scheme. For a $G$-torsor $Q$ over $\Pj^1_A$
and any two sections $s_0,s_1\in\Pj^1_A(A)$, there is an
isomorphism
\[
 s_0^*Q\simeq s_1^*Q.
\]
\end{theorem}

\section{Completing a family over a relative curve}\label{sec:completion}

This section containes the key construction, \Cref{prop:graph}, that will complete the proof.

We begin with the following generalization of Milnor's patching theorem to torsors.

\begin{proposition}\label{prop:milnor-gluing}
Let $A_1\to A_0\leftarrow A_2$ be a diagram of $R$-algebras
in which one of the maps is surjective, and set
$A=A_1\times_{A_0}A_2$.
For every algebraic $R$-stack $\mathcal X$ with
quasi-separated diagonal, restriction induces an equivalence
\[
 \mathcal X(A)\xrightarrow{\sim}
 \mathcal X(A_1)\times_{\mathcal X(A_0)}\mathcal X(A_2).
\]
In particular, if $G$ is a flat, locally finitely presented,
quasi-separated group algebraic space over $R$, then
\[
 \mathbf{B}G(A)\xrightarrow{\sim}
 \mathbf{B}G(A_1)\times_{\mathbf{B}G(A_0)}\mathbf{B}G(A_2).
\]
\end{proposition}

\begin{proof}
We may assume that $A_1\to A_0$ is surjective.
The affine Ferrand square
\[
\begin{tikzcd}
 \Spec A_0 \arrow[r,hook] \arrow[d]
   & \Spec A_1 \arrow[d] \\
 \Spec A_2 \arrow[r,hook]
   & \Spec A
\end{tikzcd}
\]
is a geometric pushout. By
\cite{AHHLR24}*{Theorem~4.2 and Lemma~4.5},
it is a $2$-pushout in the $2$-category of algebraic stacks
with quasi-separated diagonals; see also  Temkin--Tyomkin \cite{TT16}*{Theorem~4.3} for the case of algebraic spaces. Its universal property gives
the first equivalence.
For the second, $\mathbf{B}G$ is algebraic by
\cite{SP}*{Tag~06PL}, and its diagonal is quasi-separated,
being fppf-locally a base change of $G\to\Spec R$.
\end{proof}

\begin{remark}\label{rem:milnor-gluing-nonflat}
Without flatness, the torsor equivalence can fail even for
affine group schemes of finite presentation.
Let $k$ be a field, set $R=k[t]$, and consider
\[
 G=\ker\bigl(t:\mathbb G_{a,R}\to\mathbb G_{a,R}\bigr)
   =\Spec R[z]/(tz),
 \qquad
 B=R\times_k R,
\]
where both maps $R\to k$ send $t$ to zero and $B$ has
the diagonal $R$-algebra structure.
For $C=R$ or $B$, multiplication by $t$ is injective
on every flat $C$-algebra $D$, so $G(D)=0$.
Thus every fppf $G_C$-torsor is uniquely trivial:
all transition functions on a trivializing cover are the identity.
On the other hand, $G_k=\mathbb G_{a,k}$, whence
\[
 \mathbf{B}G(B)\simeq *,
 \qquad
 \mathbf{B}G(R)\times_{\mathbf{B}G(k)}\mathbf{B}G(R)\simeq k_{\mathrm{disc}}.
\]
Here $c\in k$ records the translation identifying the two
trivial torsors over $k$. Only $c=0$ is effective.
\end{remark}

The following extension result by dilation is crucial to \Cref{prop:graph}.
\begin{lemma}\label{lem:dilation}
Let $A$ be a ring, let $f\in A$, and let $G$ be a flat,
locally finitely presented, quasi-separated group algebraic
space over $A$.
Let $Q$ be a $G$-torsor over $A_f[u]$, equipped with
a trivialization $e_0$ at $u=0$.
There exists an integer $m\geq0$ such that, for every
$r\geq1$ and every $N\geq m+r$, there are a $G$-torsor
$F_N$ over $A[v]$ and an isomorphism
\[
 \alpha_N:(F_N)_{A_f[v]}\xrightarrow{\sim}
 Q\times_{\Spec A_f[u],\,u\mapsto f^Nv}\Spec A_f[v].
\]
Moreover, $F_N$ admits trivializations at $v=0$ and over
$(A/f^r)[v]$ that agree over $A/f^r$, and the former
corresponds to $e_0$ under $\alpha_N$.
\end{lemma}

\begin{proof}
Set $C=A_f[u]\times_{A_f}A$, using evaluation at $u=0$.
By Proposition~\ref{prop:milnor-gluing}, the trivialization
$e_0$ glues $Q$ and the trivial torsor over $A$
to a $G$-torsor $P$ over $C$.

The maps
$
 A[w_j]\longrightarrow C,
 a\longmapsto(a/1,a),
 w_j\longmapsto(f^{-j}u,0)
$
identify
\[
 C\simeq\varinjlim_{j\geq0}A[w_j],
 \qquad w_j\longmapsto f w_{j+1}.
\]
Indeed, the degree-zero part remains $A$, whereas each
positive-degree part is localized at $f$.
Since $\mathbf{B}G$ is locally of finite presentation over $A$,
it is limit preserving \cite[Tag~0CMX]{SP}.
Hence $P$ descends to a $G$-torsor $F$ over some
$A[w_m]$, together with an isomorphism $F_C\simeq P$.
Writing $w=w_m$, restriction along $C\to A$ supplies
a trivialization $e$ of $F|_{w=0}$, while localization gives
\[
 F_{A_f[w]}\simeq
 Q\times_{\Spec A_f[u],\,u\mapsto f^mw}\Spec A_f[w],
\]
compatibly with $e$ and $e_0$.

For $N\geq m+r$, define
$
 F_N=
 F\times_{\Spec A[w],\,w\mapsto f^{N-m}v}\Spec A[v].
$
Its localization is the required pullback of $Q$.
Modulo $f^r$, the defining substitution factors through
$w=0$. Thus $e$ induces both asserted trivializations,
with the required compatibility.
\end{proof}

To apply dilation along the boundary of a curve, we need two
disjoint divisors: one containing the marked section and the
finite exceptional locus, the other containing the boundary.

\begin{lemma}\label{lem:disjoint-divisors}
Let $R$ be a Noetherian semilocal ring. Let
$C\subset\A^1_R$ be an affine open subscheme with a section
$s\in C(R)$, and let $Z\subset C$ be an $R$-finite closed
subscheme. There are disjoint relative effective Cartier
divisors $H,D\subset\Pj^1_R$, finite locally free over $R$,
such that
\[
 Z\cup s\subset H\subset C,\qquad
 |\Pj^1_R\setminus C|\subset |D|,\qquad
 \infty_R\subset D,
\]
where the last inclusion is scheme-theoretic. Moreover, $H$
and $D$ are defined by homogeneous polynomials $F_H$ and
$F_D$ of positive degrees $h$ and $d$, respectively.
\end{lemma}

\begin{proof}
Put $X=\Pj^1_R$, $T=Z\cup s$, and let $B$ be the closed
union of $(X\setminus C)_{\mathrm{red}}$ with $\infty_R$.
The reduced boundary $(X\setminus C)_{\mathrm{red}}$ is proper and quasi-finite over $R$,
since each fibre avoids $s$, hence finite. Thus $B$ is
$R$-finite, has support $X\setminus C$, and contains
$\infty_R$ scheme-theoretically even when $R$ is nonreduced.
The $R$-finite subscheme $T$ is disjoint from $B$.

For $n \gg 0$, Serre
vanishing gives a surjection
\[
 H^0(X,\mathcal I_T(n))\to
 H^0(B,\OO(n)|_B).
\]
Since $B$ is affine and semilocal, lift a generator of
$\OO(n)|_B$ to $F_H\in H^0(X,\mathcal I_T(n))$, and let
$H$ be its zero scheme. Then $T\subset H$ and
$H\cap B=\varnothing$. Since $F_H$ generates $\OO(n)$
along $\infty_R$, it is nonzero on every fibre. By the
fibrewise Cartier criterion, \SP{062Y}, $H$ is a
relative effective Cartier divisor. It is finite, being proper
over $R$ and closed in $\A^1_R$, and locally free of rank
$n$, since its fibres have degree $n$.

Likewise, for $m \gg 0$,
lift a generator through the surjection
\[
 H^0(X,\mathcal I_B(m))\to
 H^0(H,\OO(m)|_H)
\]
to obtain $F_D$ with zero scheme $D$ satisfying
$B\subset D$ and $D\cap H=\varnothing$.
Since $F_D$ generates $\OO(m)$ along $s\subset H$, the
same criterion makes $D$ a relative effective Cartier divisor.
It is proper over $R$ and closed in the affine open
$D_+(F_H)=X\setminus H$, hence finite locally free.
Taking $h=n$ and $d=m$ gives the assertion.
\end{proof}

Now we come to the key construction of this paper. We will choose a graph that passes through the zero section over
the exceptional locus and approaches the infinity section near
the boundary. The exponent in its equation will provide the
dilation required by Lemma~\ref{lem:dilation}.

\begin{proposition}\label{prop:graph}
Let $R$ be a Noetherian semilocal ring, and let $G$
be a flat,
locally finitely presented, quasi-separated group algebraic space over $R$.
Let $C\subset\A^1_R$, $s\in C(R)$, and $Z\subset C$ be
as in Lemma~\ref{lem:disjoint-divisors}. Suppose that $Q$ is a
$G$-torsor on an open subscheme $V\subset\Pj^1_C$ containing
\[
 \Pj^1_{C\setminus Z},\qquad 0_C,\qquad\infty_C,
\]
and equipped with isomorphisms $Q|_{0_C}\simeq E_C$ and
$Q|_{\infty_C}\simeq1$. Then there is a $G$-torsor
$\mathcal F$ over $\Pj^1_R$ such that
\[
 s^*\mathcal F\simeq s^*E_C,\qquad
 \infty_R^*\mathcal F\simeq1.
\]
If $G$ is reductive, then $s^*E_C$ is trivial.
\end{proposition}

\begin{proof}
Choose $H,D,F_H,F_D,h,d$ as in
Lemma~\ref{lem:disjoint-divisors}, and write $t$ for the
auxiliary coordinate of $\Pj^1_C$, to distinguish it from the
coordinate on the curve. For $N>0$, we consider the morphism
\[
 j_N=[F_H^{dN}:F_D^{hN}]:\Pj^1_R\to\Pj^1_R
\]
that sends $H$ to $0_R$ and $D$ to $\infty_R$, scheme-theoretically.

Put $C_0=\Pj^1_R\setminus D\subset C$. The graph of $j_N$
over $C_0$ lies in $V$: over $C\setminus Z$, the open
$V$ contains the projective line, and over $Z\subset H$,
the graph factors through $0_C$. Pulling back $Q$ along
this graph gives a torsor $E_N$ over $C_0$. Since
$s\subset H\subset C_0$, the prescribed comparison at zero
gives
$
 s^*E_N\simeq s^*E_C.
$

Write  $\Spec A:=\Pj^1_R\setminus H$,
and put
$
 f=F_D^h/F_H^d\in A.
$
Then $f$ is a nonzerodivisor as it defines $hD$, and
\[
 \Spec A_f=\Pj^1_R\setminus(H\cup D)
          =(\Pj^1_R\setminus H)\cap C_0.
\]
This open lies in $C\setminus Z$, so the restriction of $Q$
is defined on the projective line over $A_f$. Its restriction
to the infinity chart, with coordinate $u=t^{-1}$, is a
$G$-torsor $P$ over $A_f[u]$ with the given trivialization
at $u=0$. On $\Spec A_f$, the graph is given by
$
 u=F_D^{hN}/F_H^{dN}=f^N.
$
For sufficiently large $N$, Lemma~\ref{lem:dilation} extends
$P(u=f^Nv)$ to a torsor over $A[v]$ with a trivialization
modulo $f$. Its restriction at $v=1$ gives a torsor $F_N$
over $A$, a trivialization modulo $f$, and an isomorphism
$
 (F_N)_{A_f}\simeq(E_N)_{A_f}.
$

Since $H\cap D=\varnothing$, the opens
$\Pj^1_R\setminus H$ and $C_0$ cover $\Pj^1_R$.
Gluing $F_N$ and $E_N$ along this isomorphism gives a
$G$-torsor $\mathcal F$ with $s^*\mathcal F\simeq s^*E_C$.
The section $\infty_R$ is contained in $D\subset\Spec A$,
so its map $A\to R$ kills $f$. Thus the trivialization
modulo $f$ induces one of $\infty_R^*\mathcal F$.
If $G$ is reductive, sectionwise constancy
(Theorem~\ref{thm:sectionwise-constancy}) identifies
$s^*\mathcal F$ with $\infty_R^*\mathcal F$, so
$s^*E_C$ is trivial.
\end{proof}

\section{Finite \'{e}tale divisors on the projective line and power pullback}\label{sec:power-maps}

Let $G$ be a simple, simply
connected group scheme over a semilocal ring $A$. The main result of this section is \Cref{prop:power}, which states that, after pullback along a power map $\varphi_d: \bP^1_A \to \bP_A^1$, $[x:y]\mapsto [x^d:y^d]$, every $G$-torsor over the projective line $\bP_A^1$ that is trivial along the infinity section is trivial away from a finite \'etale divisor contained in $\bG_{m, A}$.

We begin with the following modification argument.

\begin{lemma}\label{lem:etale-modification}
Let $A$ be a semilocal ring, let $G$ be a simple, simply
connected $A$-group scheme, and let $Y\subset\Pj^1_A$ be a
finite \'etale divisor cut out by a monic polynomial of positive
degree in an affine chart. Suppose that $G_Y$ is
isotropic in the sense that it has a $Y$-fibrewise proper parabolic $Y$-subgroup. Let $Q$ be a $G$-torsor over $\Pj^1_A$ such that
$Q|_Y$ is trivial and, for every maximal ideal $\fm$ of
$A$, the torsor $Q_{k_{\fm}}$ is trivial on $\Pj^1_{k_{\fm}}\setminus Y_{k_{\fm}}$.
There are a $G$-torsor $Q'$ over $\Pj^1_A$, pulled back
from $A$, and an isomorphism
\[
 Q'|_{\Pj^1_A\setminus Y}\simeq Q|_{\Pj^1_A\setminus Y}
\]
such that $Q'|_Y$ is trivial.
\end{lemma}

\begin{proof}
Write $Y=V(h)$ in an affine chart $\Spec A[t]$, and put
$B=A[t]/(h)$.
We will modify $Q$ along $Y$ to obtain the desired $Q'$ so that
it is trivial along all closed $A$-fibres.

By formal \'etaleness, the completion of $\bP_A^1$
along $Y$ has coordinate ring $B\fps{y}$, where $y:=h$. Put
$
 D=\Spec B\fps{y}.
$
By smoothness, any section of $Q|_Y$ lifts to a section of $Q|_D$, so the latter is also trivial.
For every maximal ideal $\fm$ of $A$, put
$B_{\fm}=B\otimes_A k_{\fm}$.
The ring $B$ is semilocal, and $G_B$ is isotropic, simple,
and simply connected, so the following loop-group surjectivity applies:
\[
 G(B\lps{y})/G(B\fps{y})
 \twoheadrightarrow
 \prod_{\fm\in\Max A}
 G(B_{\fm}\lps{y})/G(B_{\fm}\fps{y}),
\]
see \cite{Ces22a}*{(2) in the proof of Proposition~8.4} (the
essential input here is the Borel--Tits theorem \cite{Gil09}*{Fait~4.3, Lemme~4.5}; implicitly, it is contained in \cite{FP15}, and  also very close to \cite{Fed16}*{Proposition~7.1}). By Beauville--Laszlo type formal gluing for $G$-torsors \cite{BC22}*{Lemma~2.2.11~(b)}
or \cite{Fed16}*{Proposition~4.4}, this surjectivity implies that every $G$-torsor over $\bigsqcup_{\fm} \bP_{k_{\fm}}^1$ that is obtained by gluing $Q|_{\bigsqcup_{\fm}\bP^1_{k_{\fm}}\backslash Y_{k_{\fm}}}$ and $ Q|_{\bigsqcup_{\fm}Y_{k_{\fm}}}$ lifts to a $G$-torsor over $\bP_A^1$ that is obtained by gluing $Q|_{\bP^1_A\backslash Y}$ with $Q|_D$. Hence, by our assumptions, $Q|_{\bP^1_A\backslash Y}$ extends to a $G$-torsor $Q'$ over $\bP_A^1$ such that both $Q'|_{Y}$ and $Q'|_{\bigsqcup_{\fm}\bP^1_{k_{\fm}}}$ are trivial. 
By Lemma~\ref{prop:closed-fibre-descent}, the latter triviality implies that $Q'$
descends to $A$.
\end{proof}

The next lemma constructs the divisor to which the modification lemma
will be applied.

\begin{lemma}\label{lem:carrier}
Let $R$ be a Noetherian semilocal ring, and let $G$
be a simple, simply connected $R$-group scheme. There are a
finite faithfully flat \'{e}tale $R$-scheme $Y$ and a closed
immersion $Y\hookrightarrow\Gm_R$ over $R$ such that $G_Y$ admits a
Borel subgroup. If $R$ is connected, the immersion is defined
by a monic polynomial of degree $\rk_R(Y)>0$ with unit
constant term.
\end{lemma}

\begin{proof}
We pass to connected components to assume that $\Spec R$ is connected.
Let $\mathscr B$ be the smooth projective $R$-scheme parameterizing
Borel subgroups of $G$. It has constant relative dimension $r$ over $R$, because the types of geometric fibres of $G$ over $R$ are locally constant and $R$ is connected.  Choose an
$R$-projective embedding of $\mathscr B$. For every maximal
ideal $\fm \subset R$, the Bertini theorem
\cite{Ces22}*{Proposition~4.1.3} supplies $r$ hypersurfaces over $k_{\fm}$
whose intersection with $\mathscr B$ is nonempty, smooth,
and zero-dimensional. Part~(c) of that proposition allows
us to choose their degrees in common: successively, take a
sufficiently large common multiple of the finitely many
characteristic exponents. We lift the homogeneous coefficients
simultaneously by the Chinese remainder theorem. Then their common vanishing locus defines a faithfully flat $R$-finite \'etale closed subscheme
$Y_0:=\Spec B_0 \subset \mathscr B$ for which
$G_{Y_0}$ admits a Borel $Y_0$-subgroup: $Y_0$ is $R$-projective, and by the Jacobian criterion it is \'etale at all points over the closed points of $\Spec R$; hence it is finite \'etale and faithfully flat over $R$ (as the image of $Y_0\to \Spec R$ is clopen and $R$ is connected), and $G_{Y_0}$ admits a Borel $Y_0$-subgroup by construction.  

For each $N\geq1$,  by the Chinese remainder theorem we choose a monic
polynomial $f_N\in B_0[t]$ of degree $N$ whose reduction modulo every
maximal ideal $\fm\subset B_0$ is a product of $N$ distinct monic linear
factors if $k_{\fm}$ is infinite, and is irreducible if $k_{\fm}$ is finite.
This defines a finite \'etale cover $Y_N=\Spec B_N\twoheadrightarrow Y_0$, where
$B_N=B_0[t]/(f_N)$. By Panin's finite-field tricks
(see \cite{Ces22a}*{proof of Lemma~6.1}), for $N \gg 0$, there is no finite-field obstruction to embedding $Y_N$ into
$\Gm_R$:
for every maximal ideal $\fm\subset R$, the field factors of
$B_N\otimes_R k_{\fm}$ have degrees growing with $N$ and fixed
multiplicities if $k_{\fm}$ is finite; if $k_{\fm}$ is infinite,
the primitive element theorem gives such embeddings directly.

Fix such an $N$ and put $B=B_N$. The embeddings on the closed fibres
give units $a_{\fm}\in(B/\fm B)^\times$ with
$k_{\fm}[a_{\fm}]=B/\fm B$. Lift them simultaneously to $a\in B$ by
the Chinese remainder theorem. Since $B$ is $R$-finite, we have $a \in B^{\times}$, and Nakayama's lemma gives $B=R[a]$.
Thus $t\mapsto a$ defines a closed immersion
$Y=\Spec B\hookrightarrow\Gm_R$; the map $Y\to Y_0$ supplies a Borel
subgroup of $G_Y$.

Finally, $B$ is free of rank $d=\rk_R(Y)>0$. The characteristic
polynomial $h(t)$ of multiplication by $a$ is monic of degree $d$
with unit constant term. By Cayley--Hamilton, the map $R[t]\to B$
factors through a surjection $R[t]/(h)\to B$, which is an isomorphism
because both sides are locally free of rank $d$.
\end{proof}

\begin{proposition}\label{prop:power}
Let $R$ be a connected Noetherian semilocal ring, let $G$
be a simple, simply connected $R$-group scheme, and choose
$Y$ as in Lemma~\ref{lem:carrier}, of rank $d$. For every
semilocal $R$-algebra $A$ and every $G_A$-torsor $Q$ over
$\Pj^1_A$ with a given trivialization at the infinity section, the torsor
\[
 Q^{(d)}=\varphi_d^*Q,\qquad
 \varphi_d([x:y])=[x^d:y^d],
\]
admits a trivialization on $\Pj^1_A\setminus Y_A$ which
restricts to the given one at the infinity section.
\end{proposition}

\begin{proof}
For every maximal ideal $\fm$ of $A$, by
\cite{Gil02}*{th\'eorème ~3.8} (see also
\cite{Ans18}*{Proposition~3.4})\textcolor{red}{,} $Q_{k_{\fm}}$ is the pushout
of $\OO(-1)^\times$ along a $k_{\fm}$-cocharacter of $G_{k_{\fm}}$.
Thus $Q^{(d)}_{k_{\fm}}$ is the pushout of $\OO(-d)^\times$.
Since $Y_{k_{\fm}}$ is an effective divisor of degree $d$, $\OO(-d)|_{\Pj^1_{k_{\fm}}\backslash Y_{k_{\fm}}}$ is trivial, so that $Q^{(d)}_{k_{\fm}}$ is trivial away from $Y_{k_{\fm}}$.

Write $Y_A=\Spec B$. The ring $B$ is semilocal and $G_B$
admits a Borel subgroup. By \emph{loc. cit. }, $Q^{(d)}_{\bA^1_B}$ is trivial along all closed $B$-fibres. We may consequently apply
Lemma~\ref{lem:etale-modification} with $Y=\infty_A \subset \bP_B^1$ to obtain a $G$-torsor over $\bP_B^1$ that descends to $A$, restricts to $Q^{(d)}_{\bA^1_B}$, and
 is trivial along the infinity section, so that $Q^{(d)}_{\bA^1_B}$ is trivial. Pulling
back along the $B$-point defined by
$Y_A\hookrightarrow\A^1_A$ implies that
$Q^{(d)}|_{Y_A}$ is trivial.

By Lemma~\ref{lem:etale-modification} again, $Q^{(d)}|_{\Pj^1_A\setminus Y_A}$ extends
 to a torsor $Q'$ over
$\Pj^1_A$ that descends to $A$. Since $Y_A \cap \infty_A=\emptyset$, $Q'|_{\infty_A}=Q^{(d)}_{\infty_A}$. Thus $Q'$ is trivial and so is $Q'|_{\bP^1_A\setminus Y_A}$; adjust
a trivialization by a constant element of $G(A)$ to obtain the prescribed
trivialization at the infinity section.
\end{proof}

\section{Modification near the vertical fibres}\label{sec:vertical-modification}

To apply the relative-curve construction in \Cref{prop:graph}, we need a trivialization away from a
base divisor that avoids the generic points of the vertical fibres. The power-map construction in \Cref{prop:power} supplies such a trivialization
near a finite set of height-one primes. It remains to spread out the resulting modified
family and then extend it across the remaining points of codimension two.

\begin{lemma}\label{lem:spread-power}
Let $R$ be a regular semilocal domain, let $G$ be a simple, simply
connected $R$-group, and let $\mathfrak p_1,\ldots,\mathfrak p_s$
be distinct height-one primes of $R$, with $s\geq1$. Put
$
 A=(R\setminus\textstyle\bigcup_i\mathfrak p_i)^{-1}R.
$
Let $Q$ be a $G$-torsor over an open $V\subset\Pj^1_R$ containing
$\Pj^1_A,0_R,\infty_R$, and fix a trivialization of
$Q|_{\infty_R}$. For the divisor $Y$ and integer $d$ of
Proposition~\ref{prop:power}, there is a nonzero
$f\notin\bigcup_i\mathfrak p_i$ such that
$
 \Pj^1_{R_f}\subset\varphi_d^{-1}(V)
$
and $\varphi_d^*Q$ admits a trivialization over
$\Pj^1_{R_f}\setminus Y_{R_f}$ whose restriction to the infinity section is the
given one.
\end{lemma}

\begin{proof}
The ring $A$ is the filtered colimit of the rings $R_f$, where
$f\notin\bigcup_i\mathfrak p_i$. By properness, the image in $\Spec R$
of the complement of $\varphi_d^{-1}(V)$ is closed and avoids the
primes $\mathfrak p_i$. By prime avoidance, we may therefore choose
such an $f$ with $\Pj^1_{R_f}\subset\varphi_d^{-1}(V)$.
Proposition~\ref{prop:power}, applied to $Q_A$, gives a trivialization
of $\varphi_d^*Q_A$ on $\Pj^1_A\setminus Y_A$ with the prescribed
restriction to the infinity section. By spreading out, we may choose $f$
so that both this trivialization and its prescribed restriction at infinity
are defined over $R_f$.
\end{proof}

\begin{proposition}\label{prop:normalization}
Let $R$ be a Noetherian regular domain, and let $G$ be a reductive
$R$-group. Let $Q$ be a $G$-torsor on an
open $V\subset \Pj^1_R$ containing $0_R,\infty_R$ and every point of
codimension $\le 2$, and suppose that $\Pj^1_R\setminus V$ is $R$-finite. Fix isomorphisms
$Q|_{0_R}\simeq E$ and $Q|_{\infty_R}\simeq1$.
Suppose that $0\neq f\in R$, $\Pj^1_{R_f}\subset V$, and there
are an $R$-finite \'{e}tale divisor $Y\subset\Gm_R$ and a trivialization
\[
 \tau:Q|_{\Pj^1_{R_f}\setminus Y_{R_f}}\simeq1.
\]
Then there are an open $V'\subset \Pj^1_R$ containing $0_R,\infty_R$
and a $G$-torsor $Q'$ on $V'$, with the prescribed
identifications $Q'|_{0_R}\simeq E$ and
$Q'|_{\infty_R}\simeq1$, such that
\begin{enumerate}[label=(\roman*)]
 \item $V'$ contains every point of codimension $\le 2$;
 \item $\Pj^1_{R_f}\subset V'$ and the restriction $Q'|_{\Pj^1_{R_f}}$ is trivial;
 \item $\Pj^1_R\setminus V'$ is finite over $R$, and its image
       $Z\subset\Spec R$ is closed, lies in $V(f)$, and has
       codimension $\ge 2$;
 \item there is a specified isomorphism
       $Q'|_{V\setminus Y}\simeq Q|_{V\setminus Y}$, under which
       the identifications at $0_R$ and $\infty_R$ are the given ones.
\end{enumerate}
\end{proposition}

\begin{proof}
Put $W=(V\setminus Y)\cup \Pj^1_{R_f}$. Glue $Q|_{V\setminus Y}$
with the trivial torsor over $\Pj^1_{R_f}$ using $\tau$ along
$\Pj^1_{R_f}\setminus Y_{R_f}$, and denote the resulting torsor by
$Q_W$. Since $0_R,\infty_R\subset V\setminus Y$, this gluing
preserves the specified restrictions at both sections. It also gives a
trivialization of $Q_W|_{\Pj^1_{R_f}}$.

We have
\[
 \Pj^1_R\setminus W\subset
 (\Pj^1_R\setminus V)\ \cup\ (Y\cap \Pj^1_{R/f}).
\]
The first term has codimension $\ge 3$. The $R$-finite \'etale subscheme $Y \subset \bP_R^1$ is an
effective Cartier divisor. Thus $f$ is a
nonzerodivisor on $Y$, so $Y\cap \Pj^1_{R/f}$ is locally cut out by
a regular sequence of length two. It follows that $W$ contains
every point of codimension $\le 1$ and that $\Pj^1_R\setminus W$ is
$R$-finite.

To finish, we apply the purity theorem
\cite{CTS79}*{th\'eorème 6.13} to the semilocal ring of the generic
points of $\Pj^1_R\setminus W$ of codimension 2. By spreading out
and patching, the torsor $Q_W$ over $W$ extends to a $G$-torsor $Q'$
over some open $V'$ of $\bP_R^1$ containing all points of codimension
$\le 2$. As $\bP_R^1 \to \Spec R$ is smooth of relative dimension 1, the image of
$\bP_R^1 \backslash V'$ in $\Spec R$ has codimension $\ge 3-1=2$.
\end{proof}

\section{Relative curves and Nisnevich excision}\label{sec:curves-excision}

In the geometric approach to the Grothendieck--Serre conjecture, the first step is to build a relative
curve via the Gabber--Quillen type geometric presentation lemmas,
so that the relevant torsors can ``move'' in a family.
There has been tremendous effort to develop such presentation lemmas in mixed
characteristic, leading to partial success in attacking the
conjecture; see \cite{Fed21}*{Proposition 3.18},
\cite{Ces22a}*{Proposition 4.1}, \cite{GL23}*{Proposition 4.4 and Variant 4.8},
\cite{CF23}*{Proposition 2.3}. We will use the following version
from \emph{loc. cit.}\textcolor{red}{,} which weakens the finiteness requirement in the previous version to quasi-finiteness.

\begin{lemma}[\cite{CF23}*{Proposition~2.3}]
\label{prop:geometric-presentation}
Let $X$ be a smooth affine scheme of pure relative dimension $r>0$
over a Dedekind ring $\OO$, let $x_1,\ldots,x_n\in X$,
and let $Z\subset Y\subset X$ be closed subschemes.
Suppose that $Y$ contains no irreducible component of any
$\OO$-fibre and that $\codim_X Z\geq2$.
There are affine opens $X'\subset X$ and
$S\subset\A^{r-1}_{\OO}$, with all $x_i\in X'$, and a
smooth morphism $X'\to S$ of relative dimension one for
which $Z\cap X'$ is $S$-finite and $Y\cap X'$ is
$S$-quasi-finite.
\end{lemma}

The following result will help us to build Nisnevich squares that facilitate passage to (opens of) $\bP_A^1$ via patching.  

\begin{lemma}[\cite{CF23}*{Lemma~2.7}, \cite{GL23}*{ Variant 4.8}]
\label{lem:curve-embedding}
Let $A$ be a semilocal ring, let $C$ be a smooth affine
$A$-scheme of pure relative dimension one, and let
$T\subset C$ be an $A$-quasi-finite closed subscheme.
Suppose that $T_{k_{\fm}}$ admits a closed immersion into $\A^1_{k_{\fm}}$
for every maximal ideal $\fm \subset A$.
There are an affine open $C'\subset C$ containing the closed $A$-fibres of $T$, an affine open $W\subset\A^1_A$,
and an \'etale morphism $f:C'\to W$ that embeds
$T\cap C'$ as a closed subscheme
$T_W\subset W$, such that the following square is Cartesian:
\[
\begin{tikzcd}
 T\cap C' \arrow[r, hook] \arrow[d, "\sim"']
   & C' \arrow[d, "f"] \\
 T_W \arrow[r, hook]
   & W.
\end{tikzcd}
\]
\end{lemma}

\begin{lemma}\label{lem:whole-excision}
Let $R$ be a Noetherian ring, and let $f:C\to W$ be an \'{e}tale morphism
of affine $R$-schemes of finite type. Let $T\subset C$ be a closed subscheme that
 maps isomorphically to a closed subscheme
$T_W\subset W$, with
$
 C\times_W T_W=T.
$
Let $Z\subset Y\subset T$ be closed subschemes, and write $Z_W$
and $Y_W$ for their images in $W$. Let $G$ be an affine,
finitely presented $R$-group, and let $Q$ be a $G$-torsor over an
open $V\subset\Pj^1_C$ containing
$\Pj^1_{C\setminus Z},0_C,\infty_C$. Fix isomorphisms
\[
 Q|_{0_C}\simeq E_C,\qquad
 \beta:Q|_{\infty_C}\simeq1_C,
\]
and a trivialization $\alpha$ of $Q$ on
$\Pj^1_{C\setminus Y}$ whose restriction to the infinity section is
$\beta|_{C\setminus Y}$.
Then there are an open $V_W\subset\Pj^1_W$ containing
$\Pj^1_{W\setminus Z_W},0_W,\infty_W$ and a $G$-torsor
$Q_W$ on $V_W$, together with
\[
 (\operatorname{id}\times f)^{-1}(V_W)=V,
 \qquad
 (\operatorname{id}\times f)^*Q_W\simeq Q,
\]
such that $\alpha$ and $\beta$ descend to trivializations on
$\Pj^1_{W\setminus Y_W}$ and $\infty_W$, respectively.
In particular, $E_W=Q_W|_{0_W}$ satisfies
$f^*E_W\simeq E_C$. If $s\in C(R)$ factors through $T$, then
$
 (f\circ s)^*E_W\simeq s^*E_C.
$
Moreover, the reduced complements of $V$ and $V_W$ are
isomorphic over $R$.
\end{lemma}

\begin{proof}
Put $X'=\Pj^1_C$, $X=\Pj^1_W$, and
$p=\operatorname{id}\times f$. Endow $B=X'\setminus V$ with the
reduced structure; its support lies in
$\Pj^1_Z\subset\Pj^1_Y$, hence the inclusion
$B\subset X'$ factors through $\Pj^1_Y$. The excision assumption
gives a Cartesian isomorphism
$\Pj^1_Y\simeq\Pj^1_{Y_W}$. Let
$B_W\subset\Pj^1_{Y_W}\subset X$ be the closed subscheme
corresponding to $B$ under this isomorphism. Then
$p^{-1}(B_W)=B$ as closed subschemes. Thus
$
 V_W=X\setminus|B_W|
$
has inverse image $V$ and complement supported over $Z_W$,
and $B\simeq B_W$ over $R$.

The sections $0_W$ and $\infty_W$ are contained in $V_W$.
Indeed, any intersection with $B_W$ would lie over $Y_W$, where
the isomorphism $\Pj^1_Y\simeq\Pj^1_{Y_W}$ identifies it with
the empty intersection of $B$ with $0_C$ or $\infty_C$.
This also verifies the claimed inclusions at points over $Z_W$.

The following Cartesian square is a Nisnevich square, since $p$
is \'etale and induces an isomorphism
$V\cap\Pj^1_Y\simeq V_W\cap\Pj^1_{Y_W}$:
\[
\begin{tikzcd}
 \Pj^1_{C\setminus Y} \arrow[r, hook] \arrow[d, "\sim"']
   & V \arrow[d, "p"] \\
 \Pj^1_{W\setminus Y_W} \arrow[r, hook]
   & V_W.
\end{tikzcd}
\]
By Nisnevich patching (see, e.g., \cite{Ryd11}*{Theorem~A}), we may glue $Q$ with the trivial
$G$-torsor over $\Pj^1_{W\setminus Y_W}$ using $\alpha$ to obtain a $G$-torsor $Q_W$ over $V_W$, together with the
specified isomorphism $p^*Q_W\simeq Q$ and the trivialization on
$\Pj^1_{W\setminus Y_W}$.

Pulling the square back to $\infty_W$ and applying excision to
isomorphisms, we may glue $\beta$ with the standard trivialization
over $W\setminus Y_W$, since
$\alpha|_{\infty_{C\setminus Y}}=\beta|_{C\setminus Y}$.
Thus $\beta$ descends to a trivialization of
$Q_W|_{\infty_W}$. Restricting $p^*Q_W\simeq Q$ to the zero section gives
$f^*E_W\simeq E_C$; pulling back further along $s$ gives
$(f\circ s)^*E_W\simeq s^*E_C$.
\end{proof}

\begin{remark}\label{rem:normalize-frame}
The required compatibility of $\alpha$ with $\beta$ can always
be arranged. For an affine group $G$ and a scheme $D$, pullback
induces an isomorphism
$
 G(D)\xrightarrow{\sim}G(\Pj^1_D),
$
as follows locally from
$\Gamma(\Pj^1_D,\OO)=\Gamma(D,\OO)$.
Thus a trivialization on $\Pj^1_{C\setminus Y}$ can be adjusted by
a unique constant automorphism to have the prescribed restriction to the 
infinity section.
\end{remark}

\section{Proof of the main results}\label{sec:main-proof}

Let $E$ be a generically trivial $G$-torsor over $R$.
We prepare the partial family of Proposition~\ref{prop:partial-family}
for the relative-curve construction of Proposition~\ref{prop:graph}.

\begin{pp}[Reduction to simple, simply-connected group over a smooth base]
By \cite{CF23}*{Proposition~5.1}, we may first replace $G$ by $(G^{\mathrm{der}})^{\mathrm{sc}}$ to assume that $G$ is simply connected.  By the canonical decomposition \cite{SGA3IIInew}*{Expos\'e~XXIV, Sections~5.2--5.3, Proposition~5.10~(i)}, we can write 
\[
G=\mathrm{Res}_{R'/R}G'
\]
for a finite \'etale $R$-algebra $R'$ and a simple simply-connected $R'$-group $G'$. By the Shapiro bijection \cite{SGA3IIInew}*{Expos\'e~XXIV, Proposition~8.4}, we may then replace $R$ and $G$ by $R'$ and $G'$,
respectively, to assume that the $G$ is simple simply-connected. We decompose $\OO$ and $R$ into factors to
assume that both are domains. Localizing $\OO$ at the
nonzero contractions of the maximal ideals of $R$, or
replacing it by $\Frac\OO$ if there are none, we may
also assume that $\OO$ is a field or a semilocal Dedekind
domain with a nonempty closed fibre in $\Spec R$.

Combining Popescu's theorem \SP{07GC} with a limit argument, we reduce to the case when $R$ is the semilocal ring of a
smooth affine integral $\OO$-scheme $X$. By spreading out, we can assume that $G$ and $E$ are defined over $X$,
with $G$ simple simply connected and $E$ generically
trivial. If $\dim R\leq1$, the open in
Proposition~\ref{prop:partial-family} is all of $\Pj^1_R$,
so Theorem~\ref{thm:sectionwise-constancy} proves the result.
Hence we may assume that $\dim R\geq2$ and that $X$
has pure relative dimension $r>0$ over $\OO$.

Since $\rad(G)=1$, \cite{Gil21}*{Theorem~1.1} gives a
closed immersion $G\hookrightarrow\GL_{n,R}$; the quotient $\GL_{n,R}/G$
is affine by \cite{Alp14}*{Corollary~9.7.7} and finitely
presented. Thus the linearity hypothesis of
Proposition~\ref{prop:graph} is satisfied.
\end{pp}

\begin{pp}[A modification near the vertical fibres]
Applying Proposition~\ref{prop:partial-family} to $E$, we obtain a $G$-torsor
$Q$ over an open $V\subset\Pj^1_R$ containing $0_R$, $\infty_R$,
and all points of codimension $\le 2$, together with isomorphisms
\[
 Q|_{0_R}\simeq E,\qquad Q|_{\infty_R}\simeq1.
\]
As observed after that proposition, the complement of $V$
is finite over $R$.

When $\dim\OO=1$, let $\mathfrak p_1,\ldots,\mathfrak p_s$ be
the generic points of the nonempty closed $\OO$-fibres of $\Spec R$. When $\OO$ is a field, choose any one
height-one prime of $R$. In either case, let $A$ be the
semilocalization at these primes. Every point of $\Pj^1_A$ has
codimension $\le 2$ in $\Pj^1_R$, so $\Pj^1_A\subset V$.

Choose $Y\subset\Gm_R$ as in Lemma~\ref{lem:carrier}, and write
$d=\rk_R(Y)$. Proposition~\ref{prop:power} and
Lemma~\ref{lem:spread-power} give a nonzero
$f\notin\bigcup_i\mathfrak p_i$ such that
$\Pj^1_{R_f}\subset\varphi_d^{-1}(V)$ and $\varphi_d^*Q$ is
trivial away from $Y_{R_f}$, with its prescribed value at the infinity section.
The finite locally free map $\varphi_d$ preserves the finiteness
and the codimension bound of the complement, and restricts to the
identity on both sections. Apply Proposition~\ref{prop:normalization}
to this pullback. After replacing $(V,Q)$ by the resulting pullback,
we have
\[
 Q|_{0_R}\simeq E,\qquad Q|_{\infty_R}\simeq1,
 \qquad Q|_{\Pj^1_{R_f}}\simeq1,
\]
and the finite complement has a closed image $Z\subset V(f)$
of codimension $\ge 2$ in $\Spec R$. If $f$ is a unit,
the result follows again from Theorem~\ref{thm:sectionwise-constancy}. Otherwise,
we use its divisor in the following presentation.
\end{pp}

\begin{pp}[Passage to a relative curve]
We spread out $V$, $Q$, and the three displayed
isomorphisms to $X$. After shrinking $X$ around
$\Spec R$, we may assume that the reduced complement
$B_X$ is finite over $X$, contains neither distinguished
section, and has image of codimension $\ge 2$ in $X$.
We give this image its reduced structure, so that $B_X$
factors through it.

We may arrange that $f$ is a global nonzerodivisor on $X$, whose
divisor $Y_{\mathrm b}$ contains this image and contains no
irreducible component of any $\OO$-fibre. For the closed fibres,
this follows from the choice of $f$ at their generic points, after
removing components disjoint from $\Spec R$. The generic fibre is
avoided because $X$ is integral and $f\ne0$. To obtain a
closed subscheme $Z_{\mathrm b}\subset Y_{\mathrm b}$ to which
Proposition~\ref{prop:geometric-presentation} applies, choose in
the ideal of the image in $R$ an element $g$
avoiding the finitely many height-one primes of $f$. The pair
$(f,g)$ is a regular sequence; spread it out and take
$Z_{\mathrm b}=V(f,g)$, shrinking $X$ once more if necessary.

Proposition~\ref{prop:geometric-presentation}, applied to
$Z_{\mathrm b}\subset Y_{\mathrm b}\subset X$, gives a smooth
relative curve
\[
 X\to S,\qquad \text{ with }\, 
 S\subset\A^{r-1}_{\OO} \, \text{ an affine open},
\]
after shrinking $X$, such that $Z_{\mathrm b}$ is $S$-finite
and $Y_{\mathrm b}$ is $S$-quasi-finite. Base changing along
$\Spec R\to S$, we obtain a smooth affine $R$-curve $C$ with its
diagonal section $s$, a reductive $C$-group $\mathscr G$ and a $\sG$-torsor $E_C$
over $C$, and isomorphisms
\[
 s^*\mathscr G\simeq G,\qquad s^*E_C\simeq E.
\]
We also get closed subschemes $Z_C\subset Y_C\subset C$, with
$Z_C$ finite and $Y_C$ quasi-finite over $R$, and a
$\mathscr G$-torsor $Q_C$ over an open $V_C\subset\Pj^1_C$
containing
\[
 \Pj^1_{C\setminus Z_C},\qquad 0_C,\qquad\infty_C,
\]
such that $Q_C|_{0_C} \simeq E_C$, $Q_C|_{\infty_C} \simeq 1$, and $Q_C|_{\Pj^1_{C\setminus Y_C}} \simeq 1$.

The complement of $V_C$ is $R$-finite. In fact, the finite
$X$-scheme $B_X$ factors through
$(Z_{\mathrm b})_{\mathrm{red}}$, so it is finite over this scheme
and hence over $S$. Its base change is $R$-finite and has support
equal to $\Pj^1_C\setminus V_C$.

The two group schemes $\mathscr G$ and $G_C$ have the
same root datum on geometric fibres and are identified along
$s$. Applying \cite{Li23}*{Proposition~7.5} over the semilocal
ring of $C$ at the closed points of $Z_C\cup s$, and
spreading out, we may replace $C$ by a finite \'etale cover
of an affine neighbourhood of $Z_C\cup s$ on which
$\mathscr G\simeq G_C$, compatibly with the chosen lift
of $s$ and its prescribed isomorphism. We include $s$ in $Z_C$ and $Y_C$. By Panin's
finite-field tricks (see \cite{Ces22a}*{Lemma~6.1}),
a further finite \'etale cover of an affine neighbourhood of
$Z_C$, to which $s$ lifts, ensures that, for every maximal
ideal $\fm$ of $R$ with $k_{\fm}$ finite and for every $e\geq1$, the
number of degree-$e$ closed points of $(Z_C)_{k_{\fm}}$ is
strictly smaller than that of $\A^1_{k_{\fm}}$. We keep the full
inverse images of $Z_C$ and $Y_C$.
Shrinking $C$ around $Z_C$, we may remove the finitely
many points of the closed $R$-fibres of $Y_C$ outside
$Z_C$. The resulting closed fibres of $Y_C$ admit
closed immersions into the affine line by
\cite{CF23}*{Lemma~2.6}.

Lemma~\ref{lem:curve-embedding} now gives an \'etale
morphism to an affine open $W\subset\A^1_R$, fitting into
the Cartesian square
\[
\begin{tikzcd}
 Y_C \arrow[r, hook] \arrow[d, "\sim"']
   & C \arrow[d, "f_C"] \\
 Y_W \arrow[r, hook] & W,
\end{tikzcd}
\]
where $Y_W$ is closed in $W$. The shrinking allowed
by that lemma retains all of $Z_C$: a nonempty closed
subset of a finite $R$-scheme meets a closed $R$-fibre.
In particular, it retains $s$.

At each change of curve we pull back $V_C$, $Q_C$,
and their specified restrictions and trivializations. The
complement remains finite over $R$, since the covers
are finite and every subsequent open contains $Z_C$.

Using Remark~\ref{rem:normalize-frame}, adjust the trivialization on
$\Pj^1_{C\setminus Y_C}$ to have the prescribed value at the infinity section.
Lemma~\ref{lem:whole-excision}, with $T=Y_C\cup s$, descends the
open $V_C$, the torsor $Q_C$, and both of its restrictions to
$W$. The descended restriction at the zero section still pulls back to $E$
along the image of $s$. Thus, writing $C=W\subset\A^1_R$, we
have the data of Proposition~\ref{prop:graph}.

Proposition~\ref{prop:graph} constructs a torsor over
$\Pj^1_R$ whose restrictions along $s$ and $\infty_R$ are $E$ and $1$, respectively. Sectionwise constancy
(Theorem~\ref{thm:sectionwise-constancy}) therefore gives
$E\simeq1$, proving Theorem~\ref{thm:main}.\qed
\end{pp}

\begin{proof}[Proof of \Cref{red-models}]
We recall Guo's argument \cite{Guo22}*{Section~6}.
Passing to connected components, we may assume that $R$ is a normal domain, and put $K=\Frac R$.
Since the type of a reductive group scheme is locally constant and $G_K\simeq G'_K$, the groups $G$ and $G'$ have the same type.
Thus $E:=\underline{\mathrm{Isom}}_R(G,G')$ is an $\underline{\mathrm{Aut}}_R(G)$-torsor; see \cite{SGA3IIInew}*{Expos\'e~XXII, 2.8; Expos\'e~XXIV, Corollaire~1.9}.
By \cite{SGA3IIInew}*{Expos\'e~XXIV, Th\'eor\`eme~1.3}, there is an exact sequence of \'etale sheaves
\[
 1\to G^{\mathrm{ad}}\to\underline{\mathrm{Aut}}_R(G)
 \to\underline{\mathrm{Out}}_R(G)\to1,
\]
where $\underline{\mathrm{Out}}_R(G)$ is \'etale locally constant.

The quotient $Q:=E/G^{\mathrm{ad}}$ is an \'etale locally constant $\underline{\mathrm{Out}}_R(G)$-torsor by \cite{SGA3IIInew}*{Expos\'e~XXIV, Corollaire~1.10}.
Since $R$ is normal, restriction induces a bijection $Q(R)\simeq Q(K)$, as follows by \'etale descent from the constant case.
Choose $e\in E(K)$ and extend its image in $Q(K)$ to $q\in Q(R)$.
Then $P:=E\times_{Q,q}\Spec R$ is a $G^{\mathrm{ad}}$-torsor whose generic fibre contains $e$.
By \Cref{thm:main}, $P$ is trivial. Hence $E(R)\ne\varnothing$, so $G\simeq G'$.
\end{proof}

\end{document}